\documentclass{article}
\usepackage{amsfonts}
\usepackage{amsmath}
\usepackage{amssymb}
\usepackage{longtable}
\newcounter{excnt}
\newtheorem{thm}{Theorem}
\newtheorem{df}[thm]{Definition}
\newtheorem{rem}[thm]{Remark}
\newtheorem{cor}[thm]{Corollary}
\newtheorem{lem}[thm]{Lemma}
\newtheorem{satz}[thm]{Proposition}

\newtheorem{conj}{Conjecture}
\newenvironment{proof}%
     {\noindent {\bf Proof:}}{\nopagebreak\begin{flushright}{\bf q.e.d.}\end{flushright}}
     {\stepcounter{excnt}
      \noindent {\bf Example \arabic{excnt}} \phantom{i}}{\vspace{0.4cm}}

\DeclareMathOperator*{\lcm}{lcm}
\DeclareMathOperator*{\ord}{ord}

\title{Zeta Functions for Families of Calabi--Yau n-folds with Singularities}
\author{Anne Fr\"uhbis--Kr\"uger, Shabnam Kadir\\
        Institut f. Alg. Geometrie, Leibniz Universt\"at Hannover, Germany}

\begin{document}

\maketitle

\abstract{We consider families of Calabi--Yau $n$-folds containing singular
fibres and study relations between the occurring singularity structure and the
decomposition of the local Weil zeta-function. For $1$-parameter families,
this provides new insights into the combinatorial structure of the strong
equivalence classes arising in the Candelas--de la Ossa--Rodrigues-Villegas 
approach for computing the zeta-function. This can also be extended to 
families with more parameters as is explored in several examples, where 
the singularity analysis provides correct predictions for the changes of
degree in the decomposition of the zeta-function when passing to singular 
fibres. These observations provide first evidence in higher dimensions for
Lauder's conjectured analogue of the Clemens--Schmid exact sequence.}

\section{Introduction}

After a decade and a half of string theorists studying Calabi--Yau manifolds over fields of characteristic zero, particularly
in the context of mirror symmetry, Candelas, de la Ossa and Rodrigues-Villegas
\cite{CdOV1} began the exploration of arithmetic mirror symmetry. Calabi--Yau 
manifolds over finite characteristic thus became objects of interest to 
physicists as well as mathematicians.
After the discovery that the moduli spaces of all known Calabi--Yau 
manifolds form a web linked via conifold transitions \cite{GH}, the interest 
on the part of physicists decreased significantly concerning more 
complicated singularities which occur at other interesting points in the 
complex structure moduli space. However, newer results such as \cite{KLS} 
suggest that it might be worthwhile to reconsider this and ask questions 
such as: 
Is string theory viable on spaces with singularities with high Milnor 
numbers and even non-isolated singularities? Can the D-brane interpretation 
of conifold (i.e. ordinary double points) transitions by Greene, Strominger 
and Morrison \cite{S, GMS} be extended to what would be more complicated 
phase transitions? Questions of this type have not been considered very 
deeply yet - in part, because the study of singularities with more structure 
requires different methods. In this article, we want to start an approach in 
this direction by specifically studying properties at the singular fibres of 
families of Calabi--Yau varieties.
In \cite{KLS} the first question was addressed by finding points in the 
moduli space where the singular Calabi--Yau manifolds exhibited modularity, 
(i.e. are their cohomological $L$-series completely determined 
by certain modular cusp forms) as a consequence of the rank of certain 
motives decreasing in size at singularities. For an overview of Calabi--Yau 
modularity the reader may consult \cite{HKS}. Our approach, which enables the 
specification of exactly how much the degree of the contribution to the 
zeta function associated to each strong orbit (which in turn is directly 
related to motive rank) decreases, would further aid such investigations.\\ 

The local Weil zeta-function for certain families of 
Calabi--Yau varieties of various dimensions decomposes into pieces 
para\-me\-trized by monomials which are related to the toric data of the 
Calabi--Yau varieties \cite{CdOV1, CdOV2, CdO, K, Ka06}.  It was shown in these
papers that this decomposition points to deeper structures, since these 
monomials can also be related to the periods which satisfy Picard-Fuchs 
equations. Away from the singular fibres, this phenomenon of a link to 
$p$-adic periods was explained for one-parameter 
families using Monsky--Washnitzer cohomology in \cite{Kl}. The families 
considered there all have the property that one distinguished member of each 
family is a diagonal variety of Fermat type; these are very accessible to
explicit computations and are known to possess decompositions in terms of 
Fermat motives \cite{GY,KY}. \\ 

At certain values of the parameter, the corresponding variety becomes 
singular, and it was observed in \cite{K, Ka06} that the degree of the 
contribution to each piece decreases according to the types of singularities 
encountered in explicit examples. In order to test, whether the observations 
in \cite{K, Ka06} concerning the degenerations of the zeta functions for 
singular Calabi--Yau varieties hold more generally, we analyse the 
discriminant locus and singularity structure for general 1-parameter  
and some explicit 2-parameter families of Calabi--Yau varieties with 
distinguished fibre of Fermat-type and compare the results to the structure 
of their zeta functions. In particular, this provides strong evidence for 
conjectures connecting the numbers and types of singularities in the 
discriminant locus with certain combinatorial arguments arising in motivic and 
zeta function considerations and proves the facts for the considered cases by
a direct comparison. In all cases with isolated singularities the total 
Milnor number of the singularities is given precisely by the degeneration 
in the degree of the various parts of the zeta function. Observations
on finer combinatorial properties of the decomposition are also possible; 
for the $1$-parameter families, the decomposition of the singular locus and 
the Milnor numbers of the types of singularities occuring are reflected in 
the analysis of the structure of this degeneration. For these considerations,
the choice of using Dwork's original approach for computing the zeta-function 
was influenced by two constraints: by the presence of isolated singularities 
in the cases of interest and by the goal to also study higher-dimensional 
examples, which basically rules out explicit resolution of singularities 
in many cases due to the intrinsic complexity of the algorithm.   \\

After fixing notation and stating references for standard facts about the
local Weil zeta function at good primes in Section 2, we first analyze the 
occurring singularities in detail in Section 3. There we focus on 
combinatorial aspects in the calculations, which by themselves do not seem 
very exciting at first glance, but reoccur from a different perspective in 
the computation of the zeta functions for the corresponding singular fibres 
in the subsequent section. This correspondence is then explored further in 
Section \ref{twoparameter} for explicit examples of $2$-parameter families 
and leads 
to the conjectures at the end of the article linking the singularity 
structure and the decomposition of the zeta function. If these conjectures 
hold, then a singularity analysis in the singular fibres coupled with a 
calculation of the zeta function away from the singular fibres already 
provides a large amount of vital information on the zeta function at the 
singularities by using well established standard methods of singularity 
theory and of point counting.\\ 

The authors would like to thank the members of the Institut f\"ur
Algebraische Geometrie and the Graduiertenkolleg `Analysis, Geometrie und
Stringtheorie' for the good working atmosphere and the insightful discussions.
All computations of the singularity analysis were done in {\sc Singular}
\cite{Sing}, for the zeta-function calculations Mathematica \cite{Mth} 
was used.

\section{Facts about the Weil zeta function}

A pair of reflexive polyhedra $(\Delta,\Delta^{\ast})$ is known to give
rise to a pair of mirror Calabi--Yau families 
$(\hat{V}_{f,\Delta},\hat{V}_{f,\Delta^{\ast}})$.
In this setting, Batyrev proved that topological invariants such as the 
Hodge numbers could be written in terms of the toric combinatorial data 
given by the reflexive polytopes. For the case of families of Calabi--Yau 
varieties which are deformations of a Fermat variety, the data of the 
reflexive polytope is encoded in certain monomials. For a detailed treatment 
of toric constructions of mirror symmetric Calabi--Yau manifolds see 
\cite{Bat, CK}.

First we recall a few standard definitions: 
the arithmetic structure of Calabi--Yau varieties can be encoded
in the congruent or local zeta function. The Weil Conjectures (proven 
by Deligne \cite{Del} in 1974) show that the local zeta function is a
rational function determined by the cohomology of the variety.

\begin{df}[Local zeta function]
The local zeta function for a smooth projective variety $X$ defined over 
$\mathbb{F}_p$ is defined as follows:
\begin{equation}
 \zeta(X/\mathbb{F}_p,t):= \exp\left(\sum_{r \in
   \mathbb{N}} \#(X/\mathbb{F}_{p^r})\frac{t^r}{r}\right),
\end{equation}
where $\#(X/\mathbb{F}_{p^r})$ is the number of rational points of the variety.
\end{df}

For families of Calabi--Yau manifolds in weighted projective space the local 
zeta function can be computed in various ways, we however shall utilise exclusively, methods 
first developed by Dwork in his proof of the rationality part of the Weil 
conjectures \cite{Dwork1, Dwork2}. We thus use Gauss sums  composed of the 
additive Dwork character, $\Theta$ and the multiplicative Teichm\"{u}ller 
character, $\omega^n(x)$:
\begin{equation} G_n = \sum_{x\in \mathbb{F}_p^\ast}
\Theta(x)\omega^n(x). \label{defgauss}
\end{equation} 

When a variety is defined as the vanishing locus of a polynomial $P \in
k[X_1,\ldots,X_n]$, where $k$ is a field, a non-trivial additive
character like Dwork's character can be exploited to count points
over $k$. Since $\Theta(x)$ is a character:
\begin{equation}
\sum_{y\in k}\Theta(yP(x))=
\begin{cases}
            0      \quad\quad\quad\quad\quad\quad \mathrm{if}\,P(x)\neq0,\\
q:=\mathrm{Card}(k)\quad \mathrm{if}\,P(x)=0\;;
\end{cases}
\end{equation}
hence
\begin{equation}
\sum_{x_i\in k}\sum_{y\in k}\Theta(yP(x))=q\#(X/\mathbb{F}_{p^r})\;,
\end{equation}\\
The above equation can be expressed in terms of Gauss sums which are 
amenable to computation via the Gross-Koblitz formula \cite{GK}. All zeta function 
computations in this paper use an implementation of this method on 
Mathematica developed in \cite{K, Ka06}. In our context, the choice of 
this method was mainly influenced by the fact that it is also suitable
for treating singular Calabi--Yau varieties, whereas most other approaches 
are restricted to the non-singular case. Lauder's extension of the 
deformation method \cite{L2} to the singular case relies on the 
existence of an analogue of the Clemens-Schmid exact sequence in
positive characteristic which is currently only conjectural. 

The decomposition of the number of points and hence the zeta function 
into parts labelled by strong $\beta$-classes, $\mathcal{C}_{\beta}$, 
was shown in \cite{CdOV1, CdOV2, K, Ka06} as a direct consequence of 
these methods \footnote{These papers do not explicitly refer to`strong equivalence classes',
the term was coined later by Kloosterman in \cite{Kl}}.
$$\zeta(t,a)=\zeta_{\mathbf const}(t)
                \prod_{\mathcal{C}_{\beta}}
                       \zeta_{\mathcal{C}_{\beta}}(t,a)$$
where $\zeta_{\mathbf const}(t)$ is a simple term, which is independent of 
the parameter $a$, and the $\beta$-classes are defined as follows:

\begin{df}[Strong motivic $\beta$-equivalence classes]
For a given set of weights, ${\mathbf w}=(w_1,\ldots,w_n)$, 
$d=\sum_{i} w_i$, $w_i | d\ \forall i$,  
identify the set of all monomials with the set of all exponents 
of the monomials. We now consider a subset thereof defined as
$$\mathcal{M}:= \mathcal{M}({\mathbf w}):=
        \left\{{\mathbf x}=(x_1,\ldots,x_i,\ldots,x_n)
                          \in \prod_{i=1}^n w_i{\mathbb Z}/d{\mathbb Z}\ 
        \mid {\mathbf x}\cdot{\mathbf w}=ld ,\ l\in{\mathbb Z} \right\}.$$
It is easy to see that $0\leq l \leq n-1$.\\ 
Let $l({\mathbf x}):={\mathbf x}\cdot{\mathbf w}/d$. Given a 
$\beta \in {\mathcal M}$ with $l(\beta)=1$, we can quotient out 
the set $\mathcal{M}$ with the equivalence relation $\sim_{\beta}$ 
on monomials, where
$${\mathbf x}\sim_{\beta}{\mathbf y}\Leftrightarrow {\mathbf y}
    ={\mathbf x}+t\beta,\ t\in{\mathbb Z},$$
From now on we shall assume (unless otherwise stated) that the $i$th 
exponent of each monomial is taken $\mod \frac{d}{w_i}$.
The equivalence classes, $\mathcal{C}_{\beta}$, thus obtained shall 
be referred to as the {\bf strong} $\beta$ {\bf-equivalence classes}.
\label{strong}
\end{df}

\begin{rem}
For families of Calabi--Yau varieties which are deformations of smooth 
varieties of Fermat type the toric data is equivalent to specifying 
the monomials ${\mathbf x}\in\mathcal{M}$ for which $l({\mathbf x})=1$, 
see \cite{CK}.
\end{rem}

\begin{satz}[\cite{Kl}]
Considering smooth fibres of a $1$-parameter family of Calabi--Yau varieties,
the factor of the local zeta function associated to a $\beta$-class and
the parameter value $a$ is at worst a fractional power 
$$\zeta_{\mathcal{C}_{\beta}}(t,a)=
  \left(\frac{P(t,a)}{Q(t,a)}\right)^{\frac{r}{s}},$$
where $P(t,a)$ and $Q(t,a)$ are polynomials and $r,s\in \mathbf{Z}$.
We define 
$$deg(\zeta_{\mathcal{C}_{\beta}}(t,a))
     :=(deg(P)-deg(Q))\frac{a}{b}.$$
This degree of the factor of the zeta function associated to each 
strong ${\beta}$-class, can be computed as the number of monomials 
in the class which do not contain $\left(\frac{d}{w_i}-1\right)$ in its $i$th component. 
\end{satz}

Kloosterman's explanation of the above-stated relation  using 
Monsky-Washnitzer cohomology breaks down when the variety in question 
is singular. A key aim of this article is to explore the degenerations 
of the various pieces of the zeta function for singular fibres. More 
sophisticated theoretical tools such as limiting mixed Frobenius 
structures in rigid cohomology will be needed to explain the degenerations. 
Lauder \cite{L2} provides a preliminary exploration of this through the 
introduction of a conjectured analogue of the Clemens-Schmid exact sequence, 
but his testing ground for the conjecture mostly consists of families of 
curves.\\

In this article we are able to supplement Lauder's examples through looking 
at singularities of higher-dimensional varieties, not just curves, as curves
are prone to oversimplification due to their low dimensionality and could 
thus be misleading. All our results for $1$-parameter families are applicable 
in all dimensions. Moreover, all arguments are explicit and no step requires 
desingularization, which would effectively have blocked the simultaneous 
view in all dimensions. In this article we intentionally only provide 
phenomenological (and for $2$-parameter families also experimental) data, but
no theoretical explanation for the observed correspondences, because we see
it merely as the first step in this direction. We wish to disseminate the 
observations as soon as possible and would prefer to devote another article 
to the theoretical side in due time. 

\section{Singularity analysis for some families of 
         Fermat-type Calabi--Yau n-folds} \label{SingAnalysis}

In this section, we collect data about the discriminant and the singularities 
of the fibres. To this end, we first consider general $1$-parameter families
in detail and then proceed to general observations on $2$-parameter families
which establish the background for the explicit examples in Section \ref{twoparameter}.

\subsection{$1$-parameter families}

For the $1$-parameter families, we can explicitly specify a Gr\"obner 
Bases for the relative Tjurina ideal w.r.t. a lexicographical ordering, 
where the parameter $a$ of the family is considered smaller than any of 
the variables. As a consequence, we can specify the discriminant of the 
family, count the number of singularities in each fibre over the base 
space and determine the Milnor numbers of the occurring singularities. 
A priori this is not very interesting, but later on it will turn out 
that the same kind of combinatorial data which arise here also appear 
in the computation of the Weil zeta function at singular fibres of the
family. Moreover, we shall consider $2$-parameter families later on, which 
specialize to such $1$-parameter families, if one parameter is set to zero. 
For these considerations, we shall make use of the explicit calculations 
of this subsection. 

Before stating the result explicitly, we need to recall one small 
observation which will yield a key argument in the proof:

\begin{lem} \label{bezout}
Consider a polynomial ring $R[x]$ over some (noetherian commutative) ring 
$R$ (with unit). Let $f=A x^\alpha - C$, $g=Bx^\beta - D$ for some 
$A,B,C,D \in R$. Then the ideal 
$\langle f,g \rangle$ contains polynomials which we can symbolically write
as 
\begin{eqnarray*}
A^rD^sx^{\gcd(\alpha,\beta)}& - & C^rB^s, \\
C^{\frac{\beta}{\gcd(\alpha,\beta)}-r}B^{\frac{\alpha}{\gcd(\alpha,\beta)}-s}
      x^{\gcd(\alpha,\beta)} & - & A^{\frac{\beta}{\gcd(\alpha,\beta)}-r}
      D^{\frac{\alpha}{\gcd(\alpha,\beta)}-s}\\
  A^{\frac{\beta}{\gcd(\alpha,\beta)}}D^{\frac{\alpha}{\gcd(\alpha,\beta)}} 
      & - & 
      C^{\frac{\beta}{\gcd(\alpha,\beta)}}
      B^{\frac{\alpha}{\gcd(\alpha,\beta)}}
\end{eqnarray*} 
where $r,s$ are integers arising from the B\'ezout identity 
$r\alpha - s\beta = \gcd(\alpha,\beta)$; to avoid ambiguities, we choose 
precisely the ones arising from the extended Euclidean algorithm as either
$r$ and $s$ or as $\frac{\beta}{\gcd(\alpha,\beta)}-r$ and
$\frac{\alpha}{\gcd(\alpha,\beta)}-s$ making sure that $r$ and $s$ are
both positive integers. 
\end{lem}

\noindent
Using this, we can now state the main lemma of this section:\\

\begin{lem}
\label{lem2}
Let ${\mathcal X} \subset {\mathbb P}_{w_1,\dots,w_n}$, 
be the $1$-parameter family of Calabi--Yau varieties\footnote{Note that up to 
permutation of variables any $1$-parameter family of Calabi--Yau varieties with given zero 
fibre of Fermat type and perturbation term of weighted degree $d$ can be 
written in this form.} given by the polynomial
$$F=\left(\sum_{i=1}^{n} x_i^{\frac{d}{w_i}}\right)+
    a \cdot \prod_{i=1}^{k} x_i^{\beta_i}$$
where $d=\sum_{i=1}^n w_i = \sum_{i=1}^k \beta_i w_i$,
$\beta_i \neq 0 \forall 1 \leq i \leq k$, and $\gcd(w_1,\dots,w_n)=1$. 
Let $\gamma:=\gcd(\beta_1w_1,\dots,\beta_kw_k)$.
Then the discriminant of the family is
$$V\left(a^{\frac{d}{\gamma}}+(-1)^{\frac{d}{\gamma}-1}\frac{d^{\frac{d}{\gamma}}}{\prod_{i=1}^{k} \beta_iw_i^{\frac{\beta_iw_i}{\gamma}}}\right) 
\subset {\mathbb A}_{\mathbb C}^1.$$
In the respective fibre above each of the $\frac{d}{\gamma}$ points of the 
discriminant there are precisely 
$$\frac{\gcd(w_1,\dots,w_k)}{\prod_{i=1}^k w_i} \cdot d^{k-2} \cdot \gamma$$
singularities with local equation
$$\sum_{i=2}^k x_i^2 + \sum_{i=k+1}^n x_i^{\frac{d}{w_i}}$$
of Milnor number $\prod_{i=k+1}^{n} (\frac{d}{w_{i}}-1)$ and no further 
singularities.
\end{lem}

\begin{proof}\\
\underline{Preparations}:\\
As we are considering hypersurfaces here, the relative $T^1$ is of the
form $({\mathbb C}[a])[\underline{x}]/J$, where  
$$J=\left\langle F, \frac{\partial F}{\partial x_1}, \dots
            \frac{\partial F}{\partial x_n} \right\rangle$$
is the relative Tjurina ideal (Due to the weighted homogeneity and the 
resulting Euler relation, we can drop one of the $n+1$ generators.). More 
precisely, this ideal actually describes the relative $T^1$ of the affine 
cone over our family and we therefore need to ignore all contributions for 
which the associated prime is the irrelevant ideal. This is not difficult 
here, since intersection with any of the $k$ first coordinate hyperplanes 
immediately leads to an $\langle x_1,\dots,x_n \rangle$ primary ideal, and 
hence passage to any of the first $k$ affine charts immediately
removes precisly the unwanted part, but nothing else. As we are in 
weighted projective space and want to count singularities, our choice of the 
appropriate affine charts needs a little bit of extra caution: a priori we 
count points before the identification and thus might obtain a multiple of
the correct number. Hence the calculated number needs to be divided by
the weight of the respective variable. To simplify the presentation of the
subsequent steps, we choose the chart $x_1 \neq 0$.\\
\\
\underline{Gr\"obner Basis}:\\
Our next step is to compute a Gr\"obner basis of the relative Tjurina
ideal in this chart where $\alpha_i$ denotes $\frac{d}{w_i}$ to shorten
notation. For the structure of the final result, it turns out to be most
suitable to choose a lexicographical ordering with $x_2 > \dots > x_n > a$.
\begin{eqnarray*}
f_0 & = & \left(\sum_{j=2}^{n} x_j^{\alpha_j}\right)+ 1+ a \prod_{j=2}^{k} x_j^{\beta_j}\\
f_i & = & \alpha_i x_i^{\alpha_i -1} + 
          a \beta_i x^{\beta_i -1}\prod_{j=1 \atop j \neq i}^{k} x_j^{\beta_j}
 \;\;\;\;{\rm for}\;\;\; 2 \leq i \leq k \\
f_i & = & \alpha_i x_i^{\alpha_i -1} \;\;\;\;{\rm for}\;\;\;k+1 \leq i \leq n
\end{eqnarray*} 
As $f_0 - \sum_{i=2}^{n} \frac{1}{\alpha_i} x_i f_i 
   = a \frac{1}{\alpha_1} \left(\prod_{i=2}^{k} x_i^{\beta_i}\right) +1$, 
we may safely set 
$$h_0 =\frac{1}{\alpha_1} \left(a\prod_{i=2}^{k} x_i^{\beta_i}\right) +1$$
instead of the original $f_0$.\\ 
Forming $f_2 x_2 - \beta_2 \alpha_1 h_0$ and the s-polynomials of the pairs
$(f_2,h_0), \dots, (f_{k},h_0)$, we obtain new polynomials
$$h_i = x_i^{\alpha_i} - \frac{\beta_i \alpha_1}{\alpha_i}\;\;\; 
  2\leq i\leq k.$$
The leading monomials of these $h_i$, $2 \leq i \leq k$ and of the $f_i$, 
$k < i \leq n$, are obviously pure powers in the respective variables $x_i$. 
We shall use them later on when computing the discriminant.\\
Considering $h_0$ and $h_2$, we now apply Remark \ref{bezout} (polynomial
1 or 2 respectively) and obtain a polynomial 
$$g_2=x_2^{\gcd(\beta_2,\alpha_2)} - 
          {\underbrace{\left(\frac{\beta_i \alpha_1}{\alpha_i}\right)}_{:=c_1}}^r  
          \cdot \left(a \prod_{i=3}^{k} x_i^{\beta_i}\right)^{s}$$
for suitable exponents $r,s\in {\mathbb N}$ as specified in the remark. Please
note that the exponent of $x_2$, $\gcd(\beta_2,\alpha_2)$ can be written as
$\frac{1}{w_2} \gcd(d,\beta_2w_2)$. By polynomial 3 of the same remark
$$h_{0,new}= 
  c_1^{\frac{\beta_2}{\gcd(\alpha_2,\beta_2)}}\cdot \left(\frac{1}{\alpha_1}
  a\prod_{i=3}^{k}x_i^{\beta_i}\right)^{\frac{\alpha_2}{\gcd(\alpha_2,\beta_2)}}
           -1$$
In this expression, the use of properties of $\gcd$ shows that the exponent
of $x_3$ is of the form $\frac{d}{\gcd(d,\beta_2w_2)}$.
Reducing all of the $h_i$ by $g_2$, we obtain polynomials which no longer
depend on $x_2$, because all occurrences of $x_2$ in the $g_2$ were of the
form $x_2^{\beta_2}$. We are hence in the situation to apply Remark 
\ref{bezout} again, this time to $x_3$ and can eventually iterate the
process $k-2$ times. This leads to polynomials of the form
$$g_i=x_i^{\frac{d}{w_i} 
           \frac{\gcd(d,\beta_2w_2,\dots,\beta_iw_i)}
                {\gcd(d,\beta_2w_2,\dots,\beta_{i-1}w_{i-1})}}
     - c_i \cdot p_i(x_{i+1},\dots,x_{k})$$
for each $3 \leq i \leq k$. \\
To determine the discriminant we could now continue one step further,
eliminating $x_{k}$, but here it is easier to observe (e.g. by explicit 
polynomial division) that for any polynomial $1-p(\underline{x},a)$, also 
every polynomial $1-p(\underline{x},a)^k$ is in the ideal.
Applying this to $h_0$ and the $\frac{d}{\gamma}$-th power, where 
$$\gamma=\gcd(\beta_1w_1,\dots,\beta_kw_k)=\gcd(d,\beta_2w_2,\dots,\beta_kw_k),$$
we obtain
$$h_{k+1}= 1 - \left(\frac{1}{\alpha_1} a\prod_{i=2}^{k} 
               x_i^{\beta_i}\right)^{\frac{d}{\gamma}}.$$
But the exponents $\alpha_i$ of the leading monomials of the $h_i$ all divide 
$\beta_i \gamma$ for $2 \leq i \leq k$ by construction
which allows reduction of $h_{k+1}$ by these and leads to the claimed 
expression
$$g_{n+1} = a^{\frac{d}{\gamma}} 
            \frac{\prod_{i=1}^k (\beta_iw_i)^{\frac{\beta_iw_o}{\gamma}}}
                 {d^{\frac{d}{\gamma}}} 
          + (- 1)^{\frac{d}{\gamma}-1}.$$ \\
To finish the Gr\"obner basis calculation, let us first consider the set
of polynomials  
$S=\{h_2,\dots,h_{n},g_2,\dots,g_{k},g_{n+1}\}$. For $2 \leq i \leq k$ we 
drop $h_i$ from it, if the $x_i$-degree of $g_i$ is strictly smaller than 
the one of $h_i$, otherwise we drop $g_i$. The resulting set then contains 
$n$ polynomials of which each of the first $n-1$ has a pure power of the 
respective variable $x_i$ as leading monomial, and the last element $g_{k+1}$ 
which has a leading monomial not involving any of the $x_i$. Hence this 
set obviously forms a Gr\"obner basis of some ideal, because all 
s-polynomials vanish by the product criterion. It then remains to show that 
the original polynomials $f_0,\dots,f_n$ reduce to zero w.r.t. this set 
which can be checked by a straight forward but lengthy calculation.\\
\\
\underline{Reading off the data}:\\
It is clear that $a$ takes precisely the $\frac{d}{\gamma}$ 
values $$\sqrt[{\frac{d}{\gamma}}]
              {\frac{d^{\frac{d}{\gamma}}}
              {\prod_{i=1}^k (\beta_iw_i)^{\frac{\beta_iw_i}{\gamma}}}}
         \cdot \zeta$$ 
where $\zeta$ runs through all the $\frac{d}{\gamma}$-th roots of unity.
At each of these points in the base, we can obtain the number of singularities
by plugging in the value for $a$ into $g_{k}$ and counting solutions, 
followed by the values 
for $a$ and $x_{k}$ into $g_{k-1}$ and so on, where $x_{k+1}=\dots=x_n=0$. 
This leads to the expression
$$\frac{1}{w_2} \gcd(d,\beta_2w_2) 
  \frac{d}{w_3} \frac{\gcd(d,\beta_2w_2,\beta_3w_3)}{\gcd(d,\beta_2w_2)}
  \dots
  \frac{d}{w_k} \frac{\gcd(d,\beta_2w_2,\dots,\beta_kw_k)}
                     {\gcd(d,\beta_2w_2,\dots,\beta_{k-1}w_{k-1})}$$
for the number of singular points,
which after simplification of the expression and
multiplication by $\frac{\gcd(w_1,\dots,w_k)}{w_1}$ (to take account of the 
identification of points in weighted projective space) leads 
to the claimed number.\\
The multiplicity of each of these points is then given by the product of the
powers of the variables $x_i$ in the polynomials $h_i$, $k+1 \leq i \leq n$.
As the Gr\"obner basis generates the global Tjurina ideal of the fibre for 
each fixed value of $a$, the corresponding support describes the singular 
locus and the local multiplicity at each of the finitely many points is 
precisely the Tjurina number. By considering the corresponding local equations,
we can then check that the Tjurina number and Milnor number coincide for
the arising singularities.
\end{proof}

Considering the extreme cases of the families with the highest and lowest
numbers of singularities, we obtain:
plural
\begin{cor}
Let ${\mathcal X} \subset {\mathbb P}_{w_1,\dots,w_n}$ 
be the 1-parameter family of Calabi--Yau varieties given by the polynomial
$$F=\left(\sum_{i=1}^{n} x_i^{\frac{d}{w_i}}\right)+a \cdot \prod_{i=1}^{n} x_i$$
where $d=\sum_{i=1}^n w_i$ and $\gcd(w_1,\dots,w_n)=1$. Then the 
discriminant of the family is
$$V\left(a^d+(-1)^{d-1}\frac{d^{d}}{\prod_{i=1}^{n} w_i^{w_i}}\right) \subset 
 {\mathbb A}_{\mathbb C}^1.$$
In the respective fibre, above each of the $d$ points of the discriminant,
there are precisely 
$$\frac{d^{n-2}}{\prod_{i=1}^n w_i}$$
ordinary double points (with Milnor number $\mu=1$ and Tjurina number $\tau=1$)
and no further singularities.
\end{cor}

\begin{cor}
Let ${\mathcal X} \subset {\mathbb P}_{w_1,\dots,w_n}$ 
be the 1-parameter family of Calabi--Yau varieties given by the polynomial
$$F=\left(\sum_{i=1}^n x_i^{\frac{d}{w_i}}\right)+ax_1^{\frac{d-w_n}{w_1}}x_n$$
where $d=\sum_{i=1}^5 w_i$ and $w_1 | w_n$. Then the discriminant 
of the family is
$$V\left(a^{\frac{d}{w_n}}+ (-1)^{\frac{d}{w_n}-1} 
    \frac{d^{\frac{d}{w_n}}}{w_n(d-w_n)^{\frac{d}{w_n}-1}}\right)
    \subset {\mathbb A}_{\mathbb C}^1.$$
In the respective fibre above each of the $\frac{d}{w_n}$ points of the 
discriminant there is precisely $1$ isolated singularity of which the
local normal form (after moving to the coordinate origin) is
$$\left(\sum_{i=2}^{n-1}x_i^{\frac{d}{w_i}}\right)+x_n^2$$
with Milnor number 
$\mu =\prod_{i=2}^{n-1}\left(\frac{d}{w_i}-1\right).$
\end{cor} 

\subsection{Some particular 2-parameter families}

In this case, the Gr\"obner basis of the relative Tjurina ideal is far
too complicated to write down in general. Nevertheless, it is possible to
follow the lines of some of the calculations of the previous subsection
to specify and study the discriminant of some families. By analysis of the 
discriminant it is then possible to precisely classify the arising 
singularities in explicit families. \\
\begin{lem}
Let ${\mathcal X} \subset {\mathbb P}_{w_1,\dots,w_n}$
be the 2-parameter family of Calabi--Yau (n-2)-folds given by the polynomial
$$F=\sum_{i=1}^{n} x_i^{\frac{d}{w_i}}+ a \prod_{i=1}^n x_i 
   +bx_1^{\beta_1}x_2^{\beta_2}$$
where $d=\sum_{i=1}^n w_i$ and $\beta_1w_1+\beta_2w_2=d$. 
Then the discriminant of this 2-parameter family is reducible and its 
irreducible components can be sorted into two different kinds: 
\begin{itemize}
\item Lines $L_i$ parallel to the $a$-axis, which are 
          determined by the discriminant of  
          $$x^{\frac{d}{w_2}} + bx^{\beta_2} + 1$$ \\
\item A (possibly reducible) curve $C$ which can be specified as the 
      resultant of 
          $$a^d x_2^d - 
           \frac{d^{d-2}}{\prod_{i=3}^n w_i^{w_i}}
                         \left(\beta_1 bx_2 + \frac{d}{w_1}\right)$$
          and
          $$\frac{d}{w_2} x_2^{\frac{d}{w_2}} 
           + (\beta_2 - \beta_1) b x_2^{\beta_2} 
           - \frac{d}{w_1}.$$
\end{itemize}
\end{lem}

\begin{proof}
As before, we choose a suitable affine chart, say $x_1 \neq 0$, and fix
a lexicographical monomial ordering $x_n > \dots > x_2 > a > b$. But here
an explicit computation of a Gr\"obner basis of the Tjurina ideal cannot 
be performed in all generality. Instead, we can proceed analogous to the 
steps of the proof of Lemma \ref{lem2} and obtain the following elements 
of the ideal:
\begin{eqnarray*}
h_i & = & \frac{d}{w_i} x_i^{\frac{d}{w_i}} - \beta_1 b x_2^{\beta_2} 
               - \frac{d}{w_1} \forall 3 \leq i \leq n\\
h_2 & = & \frac{d}{w_2} x_2^{\frac{d}{w_2}} 
               + (\beta_2 - \beta_1) b x_2^{\beta_2} 
               - \frac{d}{w_1}\\
h_0 & = & a \prod_{i=2}^n x_i + \left(\beta_1 b x_2^{\beta_2} + \frac{d}{w_1}\right)
\end{eqnarray*}
As before, we can again conclude that also
$$\left(a \prod_{i=2}^n x_i\right)^d - \left(\beta_1 b x_2^{\beta_2} + \frac{d}{w_1}\right)^d$$
is in the ideal and forming a normal form w.r.t. $h_3,\dots,h_n$ then
yields
$$\left(\beta_1 b x_2^{\beta_2} + \frac{d}{w_1}\right)^{\sum_{i=3}^n w_i} \cdot
  \left(a^d x_2^d \prod_{i=3}^n w_i^{w_i} - 
          d^{d-2} \left(\beta_1 b x_2^{\beta_2} + \frac{d}{w_1}\right)^{w_1+w_2}\right)$$
At this point, we can branch our computation and consider each factor
separately.\\[0.2cm]

\underline{$g_1=\left(\beta_1 b x_2^{\beta_2} + \frac{d}{w_1}\right)$}:
Here we directly obtain
$$g_2 = g_1 + h_2 
      = \frac{d}{w_2}x_2^{\frac{d}{w_2}} + \beta_2 b x_2^{\beta_2}$$
and 
$$g_3 = \frac{w_1}{d} g_1 + \frac{w_2}{d} g_2 
      = x_2^{\frac{d}{w_2}} + bx_2^{\beta_2} + 1.$$
Therefore the resultant of $g_2$ and $g_3$ is also contained in the ideal.
On the other hand, $g_2 = x_2 \frac{\partial g_3}{\partial x_2}$ and hence
the above resultant is just the discriminant of $g_3$ by the rules for 
computing resultants and the fact that 
$Res_{x_2}(g_3,x_2)=1$.\\[0.2cm]

\underline{$g_4 = a^d x_2^d \prod_{i=3}^n w_i^{w_i} - 
          d^{d-2}\left(\beta_1 b x_2^{\beta_2} + \frac{d}{w_1}\right)^{w_1+w_2}$}:
As $g_4$ and $h_2$ are both in the ideal so is their resultant w.r.t. 
$x_2$ which describes the desired curve.
\end{proof}

On the basis of this lemma, it is now easy to treat interesting special 
cases, which we want to consider in a later section of this article, 
by a straight-forward computation. In order to treat such examples 
by the combinatorial algorithm for determining the zeta-function, the
two perturbation monomials need to be in the same strong $\beta$-orbit in
the sense that the orbit structure w.r.t. the second monomial refines the
one w.r.t. the first monomial. As this is a rather restrictive condition
on the possible choices of monomials, we only state a choice of three 
explicit examples in Section \ref{twoparameter}. 

\section{The influence of singularity data on the zeta function}

In the previous section, we analysed the singularity structure of some
$1$- and $2$-parameter families of Calabi--Yau varieties and, in particular,
the structure of the Milnor algebra which encodes cohomological information
about the singularities. Now we shift our focus to the computation of the
local zeta-function for these families and re-encounter combinatorial data
which we already saw in the previous section.\\

\begin{rem}
Recalling Definition \ref{strong} of strong motivic $\beta$-classes in 
$\mathcal{M}$, it is easy to show that each strong $ \beta$-class, 
$\mathcal{C}_{\beta}$, is a set with cardinality $d_{\beta}$, where
$$d_\beta=\lcm_{\beta_i\neq 0}(\ord(\beta_i))
         =\lcm_{\beta_i\neq 0}\left(\frac{d}{\gcd(\beta_iw_i,d)}\right)
         =\frac{d}{\gcd_{\beta_i\neq 0}(\beta_iw_i)}.$$
Hence the total number of strong $\beta$-classes, $\mathcal{O}_{\beta}$ is
$$\mathcal{O}_{\beta}
         =|\mathcal{M}|\frac{\gcd_{\beta_i\neq 0}(\beta_iw_i)}{d}.$$
\end{rem}

\begin{lem}
Let $w_1,\dots,w_n$ be a set of weights satisfying the conditions of 
Defintion \ref{strong}. The total number of elements in ${\mathcal M}$ is
$$ |{\mathcal M}| = \left(\prod_{i=1}^n \frac{d}{w_i}\right) \frac{1}{d_{\bf c}},$$
where $d_c$ denotes the cardinality of a strong ${\bf c}=(1,\dots,1)$-class.
\end{lem}

\begin{proof} 
The total number of monomials in 
$${\mathcal W}:=\prod_{i=1}^n w_i{\mathbb Z}/d{\mathbb Z}$$
is given by the product of the number of possible entries in each position,
i.e. $\prod_{i=1}^n \frac{d}{w_i}$. Modulo $d$, the weighted degree of 
an element of ${\mathcal W}$ can take any value in $\{0,\dots,d-1\}$ and the
number of elements of ${\mathcal W}$ mapping to the same class of weighted
degree modulo $d$ is precisely $\frac{1}{d} |{\mathcal W}|$.
Hence, this is the number of elements of weighted degree $0$ modulo $d$, 
i.e.
$$|{\mathcal M}| = \frac{1}{d} \prod_{i=1}^n \frac{d}{w_i}.$$
For later considerations, it will be convenient to modify this formula
slightly using that $\gcd(w_1,\dots,w_n)=1$ implies $d=d_c$, which proves 
the claimed formula.
\end{proof}

\begin{rem}
For any given $k \in \{1,\dots,n\}$, we can partition ${\mathcal M}$ into 
subsets for which the last $(n-k)$ entries coincide. A priori, there are 
$\prod_{i=k+1}^n \frac{d}{w_i}$ possibilities for the last $(n-k)$ entries.
As the front part of any element of ${\mathcal M}$, i.e. the first $k$ 
entries of the element, can only provide weighted degrees which are multiples 
of $\gcd(w_1,\dots,w_k)$ and as the weighted degree of any element of 
${\mathcal M}$ is a multiple of $d$, not all combinations of the last
$(n-k)$ entries can actually occur, but only those which themselves also provide
multiples of $\gcd(w_1,\dots,w_k)$ as weighted degree. Hence the total 
number of these subsets of ${\mathcal M}$ is
$$ \frac{1}{\gcd(w_1,\dots,w_k)} \prod_{i=k+1}^n \frac{d}{w_i}.$$
\end{rem} 

Combining these observations and the lemma, we obtain the following
result for the number of strong $\beta$-classes which share the same last $(n-k)$ entries:

\begin{cor}
Let $\beta \in {\mathcal M}$ satisfy $l(\beta)=1$ and 
$\beta_{k+1}=\dots=\beta_n=0$. Then the number of elements of ${\mathcal M}$
which share the same last $(n-k)$ entries is precisely
$$\frac{\gcd(w_1,\dots,w_k)}{d} \prod_{i=1}^k \frac{d}{w_i}$$
and the number of strong $\beta$-classes with these last $(n-k)$ entries is
$$T_{\beta}=\frac{\gcd(w_1,\dots,w_k)}{d} 
            \frac{\gcd(\beta_1w_1,\dots,\beta_kw_k)}{d} 
            \prod_{i=1}^k \frac{d}{w_i},$$
which coincides with the total number of singularities in the singular fibre
of a 1-parameter family of Fermat-type Calabi--Yau varieties with perturbation 
term $\underline{x}^{\beta}$ as considered in section \ref{SingAnalysis}.
\end{cor}

Applying this corollary to the two special cases of 1-parameter families
considered in \ref{SingAnalysis},
we find precisely the number of 
$A_1$-singularities in the case $\beta=(1,\dots,1)$ and 1 for the completion 
of the square. This establishes the first of the two correspondences, which
we discuss here. The second one is more subtle and links the Milnor number 
to the contributions of each $\beta$-class to the zeta-function. It is known, 
that among the monomials in ${\mathcal M}$ only those that do not contain any
entry of the form $ \left (\frac{d}{w_i}-1\right)$ in the $i$-th position 
should be counted when computing the degree of the associated piece of the 
zeta-function. Therefore counting the number of possibe ways of 
constructing such monomials seems a natural question to consider and leads
to the following observation:\\

\begin{lem}
Let $\beta \in {\mathcal M}$ satisfy $l(\beta)=1$,
$\beta_{k+1}=\dots=\beta_n=0$ and $\gcd(w_1,\dots,w_k)=1$. Then the number
of tuples which appear as the last $(n-k)$ entries in an element of 
${\mathcal M}$ and do not involve any entry of the form $ \left (\frac{d}{w_i}-1\right)$,
is precisely
$$ \prod_{i=k+1}^n \left (\frac{d}{w_i} -1 \right).$$
This coincides with the Milnor number of the appearing singularities
according to \ref{SingAnalysis}.
\end{lem}

\begin{proof}
As $\gcd(w_1,\dots,w_k)=1$, any weighted degree $\sum_{i=k+1}^n \alpha_i w_i$
can be completed to a multiple of $d$ by some contribution of the first
$k$ entries. Of these only the ones with $\alpha_i \neq \left(\frac{d}{w_i}-1\right)$
need to be counted which after a direct application of the inclusion-exclusion
formula yields the desired expression.
\end{proof}

Combining the result of this lemma and the preceding corollary, we see that
in the case of $\gcd(w_1,\dots,w_k)=1$ the total number of strong 
$\beta$-classes is precisely the total Milnor number.
On the other hand,
explicit computation showed that for all families of Calabi--Yau 3-folds
with one perturbation considered, the degree of the zeta-function drops by 
exactly the total Milnor number, e.g. for the case of the canonical 
perturbation, this is the total number of conifold 
singularities, when passing to a singular fibre.
We will see further 
occurrences of these coincidences in explicit examples for $2$-parameter
families in the next section.

The correspondence between the findings of the singularity analysis and the
intermediate results of the calculation of the zeta-function can be shown 
to further illuminate the internal structure of the combinatorial
objects involved. As the calculations in the general case are rather
technical and might block the view for the key observation, we only state
this for the case $\beta=(1,\dots,1)$:

\begin{rem}
By using standard facts about the $\gcd$, the cardinality of the set 
${\mathcal M}$ can also be stated as
$$|\mathcal{M}| = \prod_{i=2}^n \gcd\left( \frac{d}{w_i},
                             \frac{d}{\gcd(w_1,\ldots,w_{i-1})}\right),$$
which better reflects the combinatorial structure of ${\mathcal M}$. \footnote{
Note that this decomposition into a product holds for any ordering of the 
weights.}. Consider the first two weights $w_1$ and $w_2$. The 
${\mathbf c}$-subclasses associated to each weight have lengths 
$L_1=\frac{d}{w_1}$, $L_2=\frac{d}{w_2}$ respectively. The i-th coordinates 
of the ordered monomials in every ${\mathbf c}$-class take values in the 
range $0,1,2,\ldots,\left(\frac{d}{w_i}-1\right)$ going up by $1$ cyclically. 
The greatest common divisor of these two ${\mathbf c}$-subclass lengths, 
$g_{1,2}=\gcd\left( \frac{d}{w_1},  \frac{d}{w_2}\right)$, can be used to 
divide the ranges $0,1,2,\ldots,\left(\frac{d}{w_i}-1\right)$ into $g_{1,2}$ 
disjoint partitioning sets given by:
$$S_{i_k}=\left\{k,k+g_{1,2},k+2g_{1,2},\ldots,k+\left(\frac{L_i}{g_{1,2}}-1\right)g_{1,2}\right\},\  0\leq k\leq (g_{1,2}-1).$$
We can now divide the monomials in $\mathcal{M}$ with $i$th coordinate in 
$S_{i_k}$ ($i=1,2$) into $g_{1,2}$ distinct sets. Hence we have established 
that 
$$g_{1,2}=\gcd\left( \frac{d}{w_1},  \frac{d}{w_2}\right) \mid |\mathcal{M}|,$$
thus accounting for the first factor in the formula. Iterating this process,
we next compute the  ${\mathbf c}$-subclass length associated to the pair 
of weights $(w_1, w_2)$, which we shall label 
$L_{1,2}=\frac{d}{\gcd(w_1,w_2)}$. Then we find analogously to the previous
step:
$$g_{(1,2),3}=\gcd\left( \frac{d}{w_3},  \frac{d}{\gcd(w_1,w_2)}\right),$$
which again leads to a further partitioning. Eventually, this leads to a 
sequence of refinements of the partitioning which reflects the claimed 
expression for the number of elements in ${\mathcal M}$.
\end{rem}

\section{Examples of 2-parameter families}
\label{twoparameter}

The observations for the 1-parameter families might still be a combinatorial
conincidence, but passing to 2-parameter families where the singularity 
analysis is no longer purely combinatorial we still see the same phenomena:
The total Milnor number of a singular fibre matches the change of the degree
of the zeta-function when moving from a smooth to a singular fibre. These
are precisely the observations which one would expect if Lauder's conjecture
of an analogue to the Clemens-Schmid exact sequence holds.\\

\noindent The three considered examples are:

\subsection{A family in ${\mathbb P}_{(1,1,2,2,2)}$}

Considering the family in ${\mathbb P}_{(1,1,2,2,2)}$ given by
$$F=x^8+y^8+z^4+u^4+v^4+a\cdot xyzuv + b\cdot x^4y^4,$$
the discriminant consists of two lines $L_1=V(b-2)$ and 
$L_2=V(b+2)$ (denote $L=L_1 \cup L_2$) and the curve $C$ which possesses 
the two components $C_1=V(a^4-256b+512)$ and $C_2=V(a^4-256b-512)$. 
For the singular fibres of the family the following singularity types occur:
\begin{description}
\item[$(a,b) \in L \setminus (L \cap C)$:] 
           4 singularities of type $T_{4,4,4}$ ($\mu = 11$)
\item[$(a,b) \in C \setminus (C \cap L)$:] 
           64 ordinary double points
\item[$(a,b) \in L \cap C$, $a\neq 0$:] 
           4 singularities of type $T_{4,4,4}$ ($\mu = 11$) and \\
           64 ordinary double points\\
           (transversal intersections of the components of the discriminant)
\item[$(0,b) \in L \cap C$:] 
           4 singularities with local normal form $x^2+z^4+u^4+v^4$\\ 
           ($\mu=27$)\\
           (higher order contact of the components of the discriminant)
\end{description}

When computing the zeta function, we see the following degrees of the 
contributions depending on the considered fibre of the family. The
contributions are labeled by the respective $(1,1,1,1,1)$-classes; classes
only differing by a permutation of entries are collected in one line\footnote{
We list the number of permutations in the column labeled 'Perm.'}:

\begin{center}
\begin{longtable}{|c|c|c|c|c|c|}\hline
\multicolumn{6}{|c|}{Degree of Contribution $R_{\textbf{v}}(t)$ According to Singularity} \\ \hline
Monomial $\textbf{v}$ & Perm. & Smooth & 64 $A_1$ & 4 $T_{4,4,4}$ & Both \\
 & & & with  & with & or \\
 & & & $\mu_{X,x}=1$ & $\mu_{X,x}=11$ & 4 $\mu_{X,x}=27$ \\
\hline
                      & & $\mu_X=0$ & $\mu_X=64$ & $\mu_X=44$ & $\mu_X=108$ \\
\hline\hline
(0,0,0,0,0) & 1 & 6 & 5 & 4 & 3 \\
(0,2,1,1,1) & 2 & 4 & 3 & 2 & 1 \\
(6,2,0,0,0) & 1 & 4 & 3 & 2 & 1 \\ \hline
(0,0,0,2,2) & 3 & 4 & 3 & 3 & 2 \\
(2,0,1,3,3) & 6 & 2 & 1 & 1 & 0 \\
(4,0,2,0,0) & 3 & 4 & 3 & 3 & 2 \\ \hline
(0,0,2,1,1) & 3 & 3 & 2 & 2 & 1 \\
(6,0,1,0,0) & 6 & 3 & 2 & 2 & 1 \\
(0,4,0,3,3) & 3 & 4 & 3 & 3 & 2 \\
(4,0,1,1,0) & 3 & 4 & 3 & 3 & 2 \\
(2,0,3,0,0) & 6 & 3 & 2 & 2 & 1 \\
(2,2,1,1,0) & 3 & 3 & 2 & 2 & 1 \\\hline
(0,0,3,1,0) & 6 & 2 & 1 & 2 & 1 \\
(2,0,2,1,0) & 12& 2 & 1 & 2 & 1 \\
(4,0,2,3,1) & 6 & 0 &-1 & 0 &-1 \\ 
\hline\hline
degree:     &   &168&104&124& 60\\ 
\hline
degree change:&   &   & 64& 44&108\\\hline
\end{longtable}
\end{center}
The coincidence of the total Milnor number with the total drop in degree
as evident in this table, provides experimental evidence for
Lauder's conjecture.\\

For this first example of an particular family, we also provide the explicit zeta-function in 
one case, to justify the omission of this data in the later examples. 
Zeta function data is too richly detailed for the chosen focus of the article.
For $p=7$ and a fibre of the family with 4 $T_{4,4,4}$ singularities, the
zeta-function of our family has the following contributions:

\begin{longtable}{|c|c|c|}\hline
Monomial $\textbf{v}$ & Contribution & Power $\lambda_{\textbf{v}}$ \\
\hline\hline
(0,0,0,0,0) & $(1+18t+2.41pt^2+18p^3t^3+p^6t^4)$ & 1 \\
(0,2,1,1,1) & $(1-pt)(1+pt)$ & 2  \\
(6,2,0,0,0) & $(1-2pt+p^3t^2)$ & 1  \\ \hline
(0,0,0,2,2) & $(1+pt)(1+2pt+p^3t^2)$ & 3  \\
(2,0,1,3,3) & $[(1-pt)(1+pt)]^{\frac{1}{2}}$ & 6 \\
(4,0,2,0,0) & $(1+pt)(1+2pt+p^3t^2)$ & 3 \\ \hline
(0,0,2,1,1) & $[(1+p^3t^2)(1-pt)(1+pt)]^{\frac{1}{2}}$ & 3  \\
(6,0,1,0,0) & $[(1-2pt+p^3t^2)(1+2pt+p^3t^2)]^{\frac{1}{2}}$ & 6  \\
(0,4,0,3,3) & $[(1+p^3t^2)^2(1-pt)(1+pt)]^{\frac{1}{2}}$ & 3  \\
(4,0,1,1,0) & $[(1+p^3t^2)^2(1-pt)(1+pt)]^{\frac{1}{2}}$ & 3  \\
(2,0,3,0,0) & $[(1-2pt+p^3t^2)(1+2pt+p^3t^2)]^{\frac{1}{2}}$& 6  \\
(2,2,1,1,0) & $[(1+p^3t^2)(1-pt)(1+pt)]^{\frac{1}{2}}$ & 3 \\\hline
(0,0,3,1,0) & $(1-pt)(1+pt)$ & 6  \\
(2,0,2,1,0) & $[(1-2pt+p^3t^2)(1+2pt+p^3t^2)]^{\frac{1}{2}}$ & 12  \\
(4,0,2,3,1) & $1$ & 6  \\ \hline
\end{longtable}

Note that the second roots arise from the algorithmic computation of the
zeta function, but never occur in the final result, because the corresponding
contributions always arise in pairs.
 
\subsection{A family in ${\mathbb P}_{(1,1,2,2,6)}$}

Considering the family in ${\mathbb P}_{(1,1,2,2,6)}$ given by
$$F=x^{12}+y^{12}+z^6+u^6+v^2+a\cdot xyzuv + b\cdot x^6y^6,$$
the discriminant consists of two lines $L_1=V(b-2)$ and $L_2=V(b+2)$ (denote 
$L=L_1 \cup L_2$) and the curve $C$ which possesses 
the two components $C_1=V(a^6-1728b+3456)$ and $C_2=V(a^6-1728b-3456)$. 
For the singular fibres of the family the following singularity types occur:
\begin{description}
\item[$(a,b) \in L \setminus (L \cap C)$:] 
           6 singularities of type $T_{2,6,6}=Y^1_{2,2}$ ($\mu = 13$)
\item[$(a,b) \in C \setminus (C \cap L)$:] 
           72 ordinary double points
\item[$(a,b) \in L \cap C$, $a\neq 0$:] 
           6 singularities of type $T_{2,6,6}$ ($\mu = 13$) and \\
           72 ordinary double points\\
           (transversal intersections of the components of the discriminant)
\item[$(0,b) \in L \cap C$:] 
           6 singularities with local normal form $x^2+z^6+u^6+v^2$ \\
           ($\mu=25$)\\
           (higher order contact of the components of the discriminant)
\end{description}

Here the contributions to the factors of the zeta-function are the following:\\

\begin{longtable}{|c|c|c|c|c|c|}\hline
\multicolumn{6}{|c|}{Degree of Contribution $R_{\textbf{v}}(t)$ According to Singularity} \\ \hline
Monomial $\textbf{v}$ & Perm. & Smooth & 72 $A_1$ & 6 $T_{2,6,6}$ & Both \\
 & & & with  & with & or \\
 & & & $\mu_{X,x}=1$ & $\mu_{X,x}=13$ & 6 $\mu_{X,x}=25$ \\
\hline
                      & & $\mu_X=0$ & $\mu_X=72$ & $\mu_X=78$ & $\mu_X=150$ \\
\hline\hline
(0,0,0,0,0)  & 1 & 6 & 5 & 4 & 3\\
(11,1,0,0,0) & 2 & 4 & 3 & 2 & 1\\
(10,2,0,0,0) & 2 & 6 & 5 & 4 & 3\\
(9,3,0,0,0)  & 1 & 4 & 3 & 2 & 1\\\hline
(10,0,0,1,0) & 4 & 4 & 3 & 3 & 2\\
(9,1,0,1,0)  & 4 & 3 & 2 & 2 & 1\\
(8,2,0,1,0)  & 2 & 4 & 3 & 3 & 2\\
(5,5,0,1,0)  & 2 & 3 & 2 & 2 & 1\\\hline
(8,0,2,0,0)  & 4 & 6 & 5 & 4 & 3\\
(7,1,2,0,0)  & 2 & 4 & 3 & 2 & 1\\
(5,3,2,0,0)  & 4 & 4 & 3 & 2 & 1\\
(4,4,2,0,0)  & 2 & 6 & 5 & 4 & 3\\\hline
(6,0,3,0,0)  & 2 & 4 & 3 & 3 & 2\\
(5,1,3,0,0)  & 4 & 2 & 1 & 1 & 0\\
(4,2,3,0,0)  & 4 & 4 & 3 & 3 & 2\\
(3,3,3,0,0)  & 2 & 4 & 3 & 3 & 2\\\hline
(6,0,0,0,1)  & 1 & 4 & 3 & 3 & 2\\
(5,1,0,0,1)  & 2 & 4 & 3 & 3 & 2\\
(4,2,0,0,1)  & 2 & 2 & 1 & 1 & 0\\
(3,3,0,0,1)  & 1 & 4 & 3 & 3 & 2\\\hline
(2,2,0,1,1)  & 2 & 3 & 2 & 2 & 1\\
(3,1,0,1,1)  & 4 & 4 & 3 & 3 & 2\\
(10,6,0,1,1) & 4 & 3 & 2 & 2 & 1\\
(11,5,0,1,1) & 2 & 4 & 3 & 3 & 2\\\hline
(1,1,2,0,1)  & 2 & 2 & 1 & 2 & 1\\
(2,0,2,0,1)  & 4 & 2 & 1 & 2 & 1\\
(9,5,2,0,1)  & 4 & 2 & 1 & 2 & 1\\
(10,4,2,0,1) & 2 & 0 &-1 & 0 &-1\\\hline\hline
degree:      &   &254&182&176&104\\\hline
degree change:&  &   & 72& 78&150\\\hline
\end{longtable}

\subsection{A family in ${\mathbb P}_{(1,1,3,3,4)}$}

Considering the family in ${\mathbb P}_{(1,1,3,3,4)}$ given by
$$F=x^{12}+y^{12}+z^4+u^4+v^3+a\cdot xyzuv + b\cdot x^4y^4v,$$
the discriminant consists of three lines $L=V(b^3+27)$ and the curve 
$C=V(a^{12}-a^8b^4-576a^8b+512a^4b^5
           +96768a^4b^2-65536b^6-3538944b^3-47775744)$. 
For the singular fibres of the family the following singularity types occur:
\begin{description}
\item[$(a,b) \in L \setminus (L \cap C)$:]
           12 ordinary double points
\item[$(a,b) \in C \setminus ((L\cap C) \cup C_{sing})$:]
           48 ordinary double points
\item[$(0,b) \in L \cap C$:]
           12 singularities of type $X_9$\\
           (higher order contact of components of the dicriminant)
\item[$(a,b) \in L\cap C, a\neq 0$:]
           60 ordinary double points\\
           (transversal intersections of the components of the discriminant)
\item[$(a,b) \in V(9a^4-16b^4,b^3-108) \subset C_{sing}$:]
           48 $A_2$ singularities
\item[$(a,b) \in V(a^4-288b,b^3-216) \subset C_{sing}$:]
           96 $A_1$ singularities
\end{description}

\begin{center}
\begin{longtable}{|c|c|c|c|c|c|c|c|c|}\hline
\multicolumn{8}{|c|}{Degree of Contribution $R_{\textbf{v}}(t)$ According to Singularity} \\\hline
Monomial $\textbf{v}$ & Perm. & Smooth & $12$ $A_1$&$48$ $A_1$&$60$ $A_1$&$48$ $A_2$&$12$ $X_9$\\\hline\hline
(0,0,0,0,0)  & 1 & 6 & 5 & 5 & 4 & 4 & 3\\
(11,1,0,0,0) & 2 & 4 & 3 & 3 & 2 & 2 & 1\\
(10,2,0,0,0) & 2 & 4 & 3 & 3 & 2 & 2 & 1\\
(9,3,0,0,0)  & 2 & 6 & 5 & 5 & 4 & 4 & 3\\
(8,4,0,0,0)  & 2 & 6 & 5 & 5 & 4 & 4 & 3\\ 
(7,5,0,0,0)  & 2 & 4 & 3 & 3 & 2 & 2 & 1\\
(6,6,0,0,0)  & 1 & 6 & 5 & 5 & 4 & 4 & 3\\\hline
(9,0,1,0,0)  & 4 & 4 & 4 & 3 & 3 & 2 & 2\\
(8,1,1,0,0)  & 4 & 4 & 4 & 3 & 3 & 2 & 2\\ 
(7,2,1,0,0)  & 4 & 2 & 2 & 1 & 1 & 0 & 0\\ 
(6,3,1,0,0)  & 4 & 4 & 4 & 3 & 3 & 2 & 2\\ 
(5,4,1,0,0)  & 4 & 4 & 4 & 3 & 3 & 2 & 2\\
(11,10,1,0,0)& 4 & 2 & 2 & 1 & 1 & 0 & 0\\  \hline 
(6,0,2,0,0)  & 2 & 4 & 4 & 3 & 3 & 2 & 2\\
(5,1,2,0,0)  & 4 & 2 & 2 & 1 & 1 & 0 & 0\\
(4,2,2,0,0)  & 4 & 4 & 4 & 3 & 3 & 2 & 2\\
(3,3,2,0,0)  & 2 & 4 & 4 & 3 & 3 & 2 & 2\\\hline\hline 
degree:      &   &180&168&132&120& 84&72\\\hline
degree change:&  &   & 12& 48& 60& 96&108\\\hline
\end{longtable}
\end{center}

\section{Conclusion}

For one-parameter families it has been shown that the combinatorics of the 
monomial equivalence classes, which split up the zeta function, is intimately 
related to the singularity structure of the varieties. Moreover, all computed 
examples\footnote{In addition to the examples stated in this article, all 
Calabi--Yau 3-folds of Fermat-type have been systematically studied 
combinatorially from our point of view. For a number of interesting cases, 
which did not pose too many computational difficulties for the Mathematica programs, 
 the explicit zeta-functions have been determined for low primes -- all showing the same 
behaviour. We choose to include only 3 explicit examples of 2-parameter
families which already cover most of our observations, because adding further
examples would not show new phenomena.} have also shown that the change of
degree of the contribution by each labeled part of the zeta function follows
patterns of the set of strong $\beta$-classes. The total change of degree
of the zeta function upon passing to a singular fibre has been observed
as coinciding with the total Milnor number of the singular fibre.

For the more involved case of two-parameter families, it is also apparent 
from the finite number of cases computed, that the combinatorics of the 
strong equivalence classes once again seem to be reflected in the 
singularity structure.  From this arises the following conjecture, which 
strongly refines the conjectures of \cite{K, Ka06}:

\begin{conj}[Singularity -geometric/combinatoric duality]
Given a family of Calabi--Yau varieties with special fibre of Fermat type,
the total Milnor number of each arising singular fibre is expressible in 
terms of the change of the degree of the zeta function when passing to
the singular fibre. \\
The singularity structure as reflected in the relative Milnor (and Tjurina)
algebra of the family encodes information on the degree changes of factors 
of the zeta function labeled by $\beta$-classes.
\end{conj}

The degenerative properties of the zeta functions at singular points studied 
here (and the global L-series they give rise to) were recently exploited in 
\cite{KLS} in order to investigate the phenomenon of `string modularity'. 
The main result was that for several families (all containing a Fermat member 
as a special fibre), the modular form associated to part of the global 
zeta function or L-series found at a degenerate, non-Fermat point 
in the moduli space agreed with that of the motivic L-series of a 
different weighted Fermat variety. 
These pairs are called L-correlated and provide evidence that the 
conformal field theory at deformed fibres (currently difficult to define) 
are related to those of the well-defined rational conformal field theories 
of Fermat-type manifolds (Gepner models) with a completely different geometry.
Our singularity-theoretic and combinatorial results would the aid 
exploration of both finding more examples of singular members of Calabi--Yau 
families exhibiting modularity, and perhaps more L-correlated 
`string-modular' pairs.

\end{document}